\documentclass{amsart}
\usepackage[english]{babel}
\usepackage[utf8]{inputenc}
\usepackage{amsmath, amssymb,epic,graphicx,mathrsfs,enumerate}
\usepackage[all]{xy}
\usepackage{color}
\usepackage{comment}
\usepackage{enumitem}
\usepackage{hyperref}
\usepackage[noabbrev,capitalize]{cleveref}
\usepackage[]{frontespizio}

\usepackage{amsmath}
\usepackage{amsthm}
\usepackage{amssymb}
\usepackage{amsfonts}
\usepackage{latexsym}
\usepackage{epsfig}

\newtheorem{thm}{Theorem}
\newtheorem{claim}{Claim}

\numberwithin{equation}{section}

\renewcommand{\footnote}{\endnote}
\newcommand{\ignore}[1]{}\makeglossary

\begin{document}
	\bibliographystyle{amsplain}

\title{Finite groups with regular subgroup graph}
\author{Andrea Lucchini}
\address{A. Lucchini, Universit\`a di Padova, Dipartimento di Matematica ``Tullio Levi-Civita'', Via Trieste 63, 35121 Padova, Italy}
\email{lucchini@math.unipd.it}

\begin{abstract}We prove that the subgroup graph of a finite group $G$ is regular if and only if $G$ is cyclic with square-free order.
\end{abstract}

%\keywords{Profinite groups, prosolvable groups, $p$-elements.}

\thanks{Project funded by the EuropeanUnion – NextGenerationEU under the National Recovery and Resilience Plan (NRRP), Mission 4 Component 2 Investment 1.1 - Call PRIN 2022 No. 104 of February 2, 2022 of Italian Ministry of University and Research; Project 2022PSTWLB (subject area: PE - Physical Sciences and Engineering) " Group Theory and Applications"}
\maketitle

The subgroup graph $L(G)$ of a finite group $G$ is the graph whose vertices are the subgroups
of the group and two vertices, $H_1$ and $H_2,$ are connected by an edge if and only if $H_1 \leq H_2$ and there is no subgroup $K$ such that $H_1\leq K \leq H_2.$
It was conjectured by Georgiana Fasol\u{a} and Marius T\u{a}rn\u{a}uceanu that $L(G)$ is a regular graph if and only if $G$ is a cyclic group of square-free order.
In this short note we give an elementary proof of this conjecture.

\begin{thm}
	Let $G$ be a finite group. Then $\Gamma(G)$ is a regular graph if and only if $G$ is cyclic and the order of $G$ is square-free.
\end{thm}

Clearly if $G$ is cyclic and $|G|=p_1\cdots p_t$ with $p_1,\dots,p_t$ distinct primes, then every vertex of $\Gamma(G)$ has degree $t.$ So for the rest of the paper we will assume that $G$ is a finite group such that $L(G)$ is regular and we will obtain the proof of the result through successive claims.

\

Given a subgroup $H$ of $G$ we will denote with $\delta(H)$ the degree of $H$ in the subgroup graph of $G$, with $\delta_1(H)$ the number of maximal subgroups of $H$ and with $\delta_2(H)$ the number of subgroups $K$ of $G$ containing $H$ as a maximal subgroups. Clearly $\delta(H)=\delta_1(H)+\delta_2(H).$

\

For every prime $p$ dividing $|G|$, we denote by $\mathcal A_p$ the set of the subgroups of $G$ with order $p$ and by $\mathcal A$ the disjoint union
$\cup_p \mathcal A_p.$ Moreover we set $\alpha_p:=|\mathcal A_p|$ and $\alpha:=|\mathcal A|=\sum_p \alpha_p.$ Notice that $\alpha=\delta(\{1\}),$
so our assumption that $\Gamma(G)$ is regular is equivalent to saying that $\delta(H)=\alpha$ for every $H\leq G.$

\begin{claim}\label{uno}
	If $p$ is an odd prime and $g\in G$ has order a $p$-power, then $g \in N_G(A)$ for every $A\in \mathcal A.$
\end{claim}

\begin{proof}
Let $x$ be a $p$-element of maximal order with the property that $g\in \langle x\rangle.$ Assume that $H\leq G$ has the property that $\langle x\rangle$ is a maximal subgroup of $H.$ By the maximality of $x,$ $H$
is not a cyclic $p$-group, so by \cite[5,3,6]{rob} it contains a minimal subgroup $A$ which is not contained in $\langle x \rangle$. Since $\langle x \rangle$ is a maximal subgroup of $H$, we must have $H=\langle x, A \rangle.$ In other words, every subgroup $H$ of $G$ in which $\langle x\rangle$ is maximal is generated by $x$ and an element of $\mathcal A$ different from the unique subgroup of order $p$ contained in $\langle x \rangle.$ This implies $\delta_2(\langle x\rangle)\leq \alpha-1.$ On the other hand $\delta_1(\langle x\rangle)=1$ so we must have
$\delta_2(\langle x\rangle) = \alpha-1.$ This implies that $\langle x, A_1\rangle \neq \langle x, A_2\rangle$
for every pair of distinct elements $A_1, A_2$ of $\mathcal A$. In particular $A^x=A$ for every $A\in \mathcal A.$ Since $g\in \langle x\rangle,$ we conclude
$g\in N_G(A)$ for every $A\in \mathcal A.$
\end{proof}

\begin{claim}\label{due}
	Let $N$ be the subgroup of $G$ generated by the minimal subgroups of $G$ with odd order. Then $N$ is an abelian normal subgroup of $G$.
\end{claim}
\begin{proof}
	Notice that $N=\langle A\mid A \in \cup_{p\neq 2}\mathcal A_p\rangle.$ By the previous claim, if $A_1, A_2\in \cup_{p\neq 2}A_p$, then $A_1\leq N_G(A_2)$ and $A_2\leq N_G(A_1).$ Hence, if $A_1\neq A_2$ then $[A_1,A_2]\in A_1\cap A_2=1.$
\end{proof}

\begin{claim}\label{tre}
$G/N$ is a 2-group.
\end{claim}

\begin{proof}
	Since $N$ is a direct product of elementary abelian $p$-groups, the numbers of maximal and minimal subgroups of $N$ coincides. In particular $\delta_1(N)=\sum_{p\neq 2}\alpha_p$ and therefore $\delta_2(N)=\alpha_2.$ By \cref{uno}, $N\leq N_G(A)=C_G(A)$ for every $A\in \mathcal A_2.$
	Thus, for every $A\in \mathcal A_2$, $N$ is a maximal subgroup of $\langle N, A\rangle \cong N \times A$ and different choices of $A\in \mathcal A_2$ give raise to different subgroups. So from the equality $\delta_2(N)=\alpha_2$, we deduce that $N$ has index 2 in every subgroup $H$ of $G$ in which it is maximal. This implies that $G/N$ is a 2-group.
\end{proof}

\begin{claim}\label{quattro}Let $Q$ be a Sylow 2-subgroup of $G$. Then $Q$ is elementary abelian.
	\end{claim}
	
	\begin{proof}Let $F$ be the Frattini subgroup of $Q$. Notice that the number of minimal subgroups and maximal subgroup of $Q/F$ coincide. Denote by $t$ this number. Now assume that $H_1,\dots,H_u$ are the subgroups of $G$ in which $Q$ is maximal. For $1\leq i \leq u,$ we must have $H_i=QM_i$ with $M_i$ a nontrivial subgroup of $N$ normalized by $Q$ and minimal with respect to this property. Notice that
		if $i_1\neq i_2$ then $M_{i_1}	\cap M_{i_2}=1.$ For each $1\leq i\leq u$, let $T_i$ be a nontrivial subgroup of $M_i$ minimal for the property of being normalized by $F.$ Then $K_1=FT_1,\dots,FT_u$ are different subgroups of $G$ containing $F$ as maximal subgroups.
		But now we have $\delta(Q)=\delta_1(Q)+\delta_2(Q)=t+u$ and $\delta_2(F)\geq t+u.$ Hence the equality $\delta(Q)=\delta(F)$ implies $\delta_1(F)=1$ and consequently $F=1.$	\end{proof}

\begin{claim}
	$G$ is abelian and all its Sylow subgroups are elementary abelian.
\end{claim}
\begin{proof}
We have that $G=NQ$ with $Q$ a Sylow 2-subgroup. By the previous claim $Q$ is elementary abelian. In particular,
$Q=\langle A \mid A\in \mathcal A_2\rangle$. By Claim 1, for every $A\in \mathcal A_2$, $N\leq N_G(A)=C_G(A)$, hence
$[N,Q]=1$ and therefore $G$ is abelian.
\end{proof}

\begin{claim}
	$G$ is cyclic with square-free order.
\end{claim}

\begin{proof}
By the previous claim, it suffices to prove that every Sylow subgroup of $G$ is cyclic of prime order. Let $p$ be a prime divisor of $|G|$ and suppose by contradiction that the Sylow $p$-subgroup $P$ of $G$ has order $p^d$ with $d\geq 2.$ Let $X$ be a maximal subgroup pf $G$. Then $$\delta(X)=\delta_1(X)+\delta_2(X)=\left(\frac{p^{d-1}-1}{p-1}\right)+\left(1+\sum_{q\neq p}\alpha_p(q)\right).$$
Since $$\delta(\{ 1\})=\sum_r \alpha_r=\frac{p^d-1}{p-1}+\sum_{q\neq p}\alpha_p(G)$$
it follows
$$\frac{p^{d-1}-1}{p-1}+1=\frac{p^d-1}{p-1},$$
a contradiction.
\end{proof}

\end{document}